\documentclass[12pt]{amsart}

\usepackage{amssymb,latexsym}
\usepackage{graphicx,epsfig,color}

\newtheorem{theorem}{Theorem}[section]

\newtheorem{corollary}[theorem]{Corollary}
\newtheorem{lemma}[theorem]{Lemma}

\numberwithin{equation}{section}

\date{}

\begin{document}

\author[Elyor B. Dilmurodov]{Elyor B. Dilmurodov$^{1,2}$}
\title[Discrete Eigenvalues of a $2 \times 2$ Operator Matrix]
{Discrete Eigenvalues of a $2 \times 2$ Operator Matrix} \maketitle

\begin{center}
{\small $^1$Faculty of Physics and Mathematics, Bukhara State University\\
$^2$Bukhara branch of the Institute of Mathematics named after V.I.Romanovskiy\\
M. Ikbol str. 11, 200100 Bukhara, Uzbekistan\\
E-mail: elyor.dilmurodov@mail.ru}
\end{center}

\begin{abstract}
We consider a $2\times2$ block operator matrix ${\mathcal A}_\mu$ $($$\mu>0$ is a coupling constant$)$ acting in the
direct sum of one- and two-particle subspaces of a bosonic Fock space. The location of the essential spectrum of ${\mathcal A}_\mu$ is described and its bounds are estimated. It is shown that there exist the critical values
$\mu_l^0(\gamma)$ with $\gamma>0$ and $\mu_r^0(\gamma)$ with $\gamma<12$ of the coupling constant $\mu>0$ such that
for all $\gamma>0$ $(\gamma<12)$ the operator ${\mathcal A}_\mu$ with $\mu=\mu_l^0(\gamma)$ $(\mu=\mu_r^0(\gamma)$
has infinitely many eigenvalues on the l.h.s. $($r.h.s.$)$ of the its essential spectrum.
We prove that for all $\mu \not\in \{\mu_l^0(\gamma),\mu_r^0(\gamma)\}$ the operator ${\mathcal A}_\mu$
has finitely many discrete eigenvalues on the l.h.s. and r.h.s. of its essential spectrum.
\end{abstract}



\section{Introduction}\label{Introduction}

Block operator matrices are matrices the entries which are linear operators between Banach or Hilbert
spaces \cite{Jeribi, Tretter}. They arise in various areas of mathematics and its applications.
One of the important class of the block operator matrices are the energy operators of a system
of $n$ non conserved number of particles. Such systems occur widely in mathematical physics, e.g.
statistical physics \cite{Min-Sp}, solid-state physics \cite{Mog} and the theory of quantum fields \cite{Frid}.
We remark that the study of systems with a non conserved, but finite number of particles is reduced
to the study of the spectral properties of self-adjoint operators acting in "the cut" $ $ subspace ${\mathcal H}^{(n)}$,
consisting of one-particle, two-particle and $n$-particle subspaces of the Fock space \cite{Min-Sp,Mog}.

One of the most actively studied objects in operator theory, mathematical
physics and related fields, is the investigation of the number of discrete eigenvalues of the block operator
matrices in the cut subspaces of the bosonic Fock space. More exactly,
to find out whether the set of eigenvalues located
in the l.h.s. and r.h.s. of the essential spectrum is finite or
infinite. The latter result with respect to the l.h.s. is the remarkable phenomenon known as {\it
Efimov effect} in the spectral theory of the three-particle
Schr\"{o}dinger operators. This property
was discovered by V.Efimov \cite{Efim} and has been the subject
of many papers \cite{Amad-Nob, Dell-Fig-Teta, Ovch-Sig, Sob,
Tam-1, Yaf}.

The main result obtained by Sobolev \cite{Sob} (see also \cite{Tam-1}) is an asymptotics of the form
${\mathcal U}_0 |\log|\lambda||$ for the number $N(\lambda)$ of eigenvalues on the left of $\lambda$,
$\lambda<0$, where the coefficient ${\mathcal U}_0$ does not depend on the two-particle potentials $v_\alpha$
and is a positive function of the ratio $m_1/m_2$ and $m_2/m_3$ of the masses of the three particles.

In models of solid state physics \cite{GrafSchen, Mog, Yaf2000}  and also in lattice quantum field theory \cite{MalMil}, one considers discrete Schr\"{o}dinger operators which are lattice analogs of the three-particle Schr\"{o}dinger operator in the continuous case. The presence of the Efimov effect for these operators was proved  in \cite{AbdLak03,AlbLakMum04, Lak91,LakMum03}.
In \cite{AbdLak03,AlbLakMum04}, an asymptotics analogous to \cite{Sob,Tam-1} was obtained for the number of eigenvalues $N(\lambda)$.
In all above mentioned papers devoted to Efimov's effect, systems of the fixed number of particles have been considered.
The existence of this effect for the operator matrices associated with the energy operator of a system of three
non conserved number of particles was proved in \cite{AlbLakRas07,AlbLakRas07-1,LakRas03, MumRas14,MumRas14-1, MumRasTosh, Ras11}.
In particular, the asymptotic formula for the number of eigenvalues were established in \cite{AlbLakRas07,MumRas14,MumRas14-1,Ras11}.
Moreover, under some natural conditions the infiniteness of the number of
eigenvalues located respectively inside, in the gap, and below of the
bottom of the essential spectrum of the corresponding operator matrices were proved in \cite{MumRas14,MumRas14-1}.

It is remarkable that the presence of a zero energy resonance for the Schr\"{o}dinger operators is due to the two-particle
interaction operators, in particular, the coupling constant, see, e.g. \cite{AbdLak03,AlbLakMum04,LakMum03}.
For the operator matrices the role of two-particle discrete Schr\"{o}dinger operators is played by a family
of generalized Friedrichs models \cite{AlbLakRas07,AlbLakRas07-1,MumRas14,Ras11} and hence for the generalized
Friedrichs model the presence of a threshold energy resonance (consequently the existence of infinitely many
eigenvalues of corresponding operator matrices) is due to the annihilation and creation operators acting in the
bosonic Fock space.

In the present paper we are concerned with the discrete spectrum analysis for the $2 \times 2$ block operator
matrix ${\mathcal A}_\mu$ acting in the direct sum of one- and two-particle subspaces of a bosonic Fock space,
in the case, where the functions with spectral parameter $\gamma \in {\Bbb R}$
of the diagonal elements of ${\mathcal A}_\mu$ has the special forms.
We prove that there are critical values $\mu_l^0(\gamma)$ with $\gamma>0$ and $\mu_r^0(\gamma)$ with $\gamma<12$ of the parameter $\mu>0$ such that only for $\mu=\mu_l^0(\gamma)$ (resp. $\mu=\mu_r^0(\gamma)$) the
operator ${\mathcal A}_\mu$ has an infinitely many eigenvalues lying on the l.h.s. of $0$ accumulating to $0$ for all $\gamma>0$
(resp. r.h.s. of $18$ accumulating to $18$ for all $\gamma<12$). For the case $\gamma=6$ it is proved that there exist
infinitely many eigenvalues located in the both sides of the essential spectrum of ${\mathcal A}_\mu$ (so-called
{\bf two-sided Efimov's effect}). We establish the following asymptotic formulas for the number $N_{(a; b)}({\mathcal A}_\mu)$
of eigenvalue of ${\mathcal A}_\mu$ lying in $(a; b)\subset{\Bbb R} \setminus \sigma_{\rm ess}({\mathcal A}_\mu)$:
\begin{align*}
& \lim\limits_{z \nearrow 0}\frac{N_{(-\infty; z)}({\mathcal A}_{\mu_l^0(\gamma)})}{|\log|z||}={\mathcal U}_0\quad \mbox{with} \quad \gamma>0;\\
& \lim\limits_{z \searrow 18}\frac{N_{(z; +\infty)}({\mathcal A}_{\mu_r^0(\gamma)})}{|\log|z-18||}={\mathcal U}_0 \quad \mbox{with} \quad \gamma<12;
\end{align*}
for some ${\mathcal U}_0 \in (0; +\infty).$ It is shown that for $\gamma=6$ the equality $\mu_l^0(\gamma)=\mu_r^0(\gamma)$ holds and
\begin{align}\label{1}
\lim\limits_{z \nearrow 0}\frac{N_{(-\infty; z)}({\mathcal A}_{\mu_l^0(6)})}{|\log|z||}=
\lim\limits_{z \searrow 18}\frac{N_{(z; +\infty)}({\mathcal A}_{\mu_r^0(6)})}{|\log|z-18||}={\mathcal U}_0.
\end{align}

We prove the finiteness of eigenvalues of ${\mathcal A}_\mu$ with $\mu\notin\{\mu_l^0(\gamma), \mu_r^0(\gamma)\},$ lying on the l.h.s. of
$\min\sigma_{\rm ess}({\mathcal A}_\mu)$ and r.h.s. of $\max\sigma_{\rm ess}({\mathcal A}_\mu).$

Note that assertion (\ref{1}) seems to be quite new for the lattice operator matrices in a bosonic Fock space. We point out that this assertion is typical for lattice case; in fact, they do not have analogues in the continuous case.

The plan of the paper is as follows. Section 1 is the general
introduction. In Section 2, we state the problem. In Section 3, we recall some spectral properties of the
family of generalized Friedrichs models. In Section 4, we estimate the
bounds of the essential spectrum of ${\mathcal A}_\mu.$ In Section 5, we prove a realization of the Birman-Schwinger principle for ${\mathcal A}_\mu.$
In Section 6, we prove that the number of eigenvalues of the operator matrix ${\mathcal A}_\mu$ with $\mu\notin\{\mu_l^0(\gamma), \mu_r^0(\gamma)\}$ is finite. In Section 7, for the cases $\mu=\mu_r^0(\gamma)$ and $\mu=\mu_r^0(\gamma)$ an asymptotic formulas for the number of eigenvalues of
${\mathcal A}_\mu$ are obtained.

\section{Statement of the problem}\label{www2}

\indent We adopt the following conventions throughout the present paper.
Let ${\Bbb N},$ ${\Bbb Z},$ ${\Bbb R}$ and ${\Bbb C}$ be the set
of all positive integers, integers, real and complex numbers,
respectively. We denote by ${\Bbb T}^3$ the three-dimensional
torus (the first Brillouin zone, i.e., dual group of ${\Bbb
Z}^3$), the cube $(-\pi,\pi]^3$ with appropriately identified
sides equipped with its Haar measure. The torus ${\Bbb T}^3$ will
always be considered as an abelian group with respect to the
addition and multiplication by real numbers regarded as operations
on the three-dimensional space ${\Bbb R}^3$ modulo $(2 \pi {\Bbb
Z})^3.$

Let us briefly set up the problem.
Denote by $L_2({\Bbb T}^{\rm 3})$ the
Hilbert space of square integrable (complex) functions defined on
${\Bbb T}^{\rm 3}$ and $ L_2^{\rm s}(({\Bbb T}^{\rm 3})^2)$
the Hilbert space of square integrable (complex) symmetric
functions defined on $({\Bbb T}^{\rm 3})^2.$ Let ${\mathcal
H}$ be the direct sum of Hilbert spaces ${\mathcal H}_1:=L_2({\Bbb T}^{\rm
3})$ and ${\mathcal H}_2:=L_2^{\rm s}(({\Bbb T}^{\rm 3})^2),$ that
is, ${\mathcal H}:={\mathcal H}_1 \oplus {\mathcal H}_2.$
We write elements $f$ of the space ${\mathcal H}$ in the form $f=(f_1, f_2),$
where $f_i\in{{\mathcal H}_i}, i=1,2.$ The norm of ${\mathcal H}$ is given by
$$
\|f\|:=\left(\int_{{\Bbb T}^3} |f_1(p)|^2dp+\int_{({\Bbb T}^3)^2} |f_2(p,q)|^2dpdq\right)^{1/2}.
$$

The spaces ${\mathcal H}_1$ and ${\mathcal H}_2$ are usually called one- and
two-particle subspaces of a bosonic Fock space ${\mathcal F}_{\rm
s}(L_2({\Bbb T}^{\rm 3}))$ over $L_2({\Bbb T}^3),$ respectively,
where
\begin{equation*}
{\mathcal F}_{\rm s}(L_2({\Bbb T}^{\rm 3})):={\Bbb C} \oplus
L_2({\Bbb T}^{\rm 3}) \oplus L_2^{\rm s}(({\Bbb T}^{\rm 3})^2)
\oplus \ldots \oplus L_2^{\rm s}(({\Bbb T}^{\rm 3})^n) \oplus
\ldots .
\end{equation*}

Here $L_2^{\rm s}(({\Bbb T}^3)^n)$ is the Hilbert space of square integrable functions on $({\Bbb T}^3)^n,$ that are symmetric with respect to the variable $k_i, i=1,2, \ldots , n.$

It is well known that \cite{Jeribi, Tretter} every bounded linear operator can be written as a $2 \times 2$
block operator matrix if the space in which it acts is decomposed in two components.
In our case the space $\mathcal H$ has such property.
In the present paper we consider the block operator matrix
${\mathcal A}_\mu$ acting in the Hilbert space ${\mathcal H}$ given by
\begin{equation*}
{\mathcal A}_\mu:=\left( \begin{array}{cc}
A_{11} & \mu A_{12}\\
\mu A_{12}^* & A_{22}\\
\end{array}
\right)
\end{equation*}
with the entries $A_{ij}: {\mathcal H}_j \to {\mathcal H}_i$, $i \leq j$, $i,j=1,2$:
\begin{align*}
& (A_{11}f_1)(k)=w_1(k)f_1(k), \quad
(A_{12}f_2)(k)= \int_{{\Bbb T}^3} f_2(k,t)dt, \\
& (A_{22}f_2)(k,p)=w_2(k,p)f_2(k,p),\quad f_i \in {\mathcal H}_i,\quad i=1,2,
\end{align*}
where $A_{12}^*$ denotes the adjoint operator to $A_{12}$, that is,
$$
(A_{12}^*f_1)(k,p)=\frac{1}{2}(f_1(k)+f_1(p)), \quad f_1 \in {\mathcal H}_1.
$$
Here $\mu>0$ is a real positive number (coupling constant), the functions $w_1(\cdot)$
and $w_2(\cdot, \cdot)$ have the form
\begin{equation*}
w_1(k):=\varepsilon(k)+\gamma, \quad
w_2(k,p):=\varepsilon(k)+\varepsilon(\frac{1}{2}(k+p))+\varepsilon(p)
\end{equation*}
with $\gamma \in {\Bbb R}$ and the dispersion function
$\varepsilon(\cdot)$ is defined by
\begin{equation}\label{epsilon}
\varepsilon(k):=\sum_{i=1}^3 (1-\cos \, k_i),\,k=(k_1, k_2, k_3)\in {\Bbb T}^3.
\end{equation}

Under these assumptions the operator matrix ${\mathcal A}_\mu$ is bounded
and self-adjoint in the Hilbert space ${\mathcal H}$.

We remark that the operators $A_{12}$ and $A_{12}^*$ are called
annihilation and creation operators \cite{Frid}, respectively. In
physics, an annihilation operator is an operator that lowers the
number of particles in a given state by one, a creation operator is
an operator that increases the number of particles in a given state
by one, and it is the adjoint of the annihilation operator.

Our aim is to use information about the entries $A_{ij}, i,j=1,2$ to investigate various spectral properties
(related with the essential and discrete spectrum) of the operator matrix ${\mathcal A}_\mu.$

\section{Generalized Friedrichs model and its spectrum}\label{www3}
\indent In this section we study some spectral properties of the family of generalized Fridrichs models ${\mathcal A}_\mu(k),$
$k\in {\Bbb T}^{3},$ defined below, which
plays a crucial role in the study of the spectral properties of ${\mathcal A}_\mu.$

Denote by $\sigma(\cdot),$ $\sigma_{\rm ess}(\cdot)$ and
$\sigma_{\rm disc}(\cdot),$ respectively, the spectrum, the
essential spectrum, and the discrete spectrum of a bounded
self-adjoint operator.

We introduce a family of bounded self-adjoint
operators (a family of generalized Friedrichs models) ${\mathcal A}_\mu(k),$
$k\in {\Bbb T}^{3}$, acting in ${\mathcal H}_0 \oplus
{\mathcal H}_1$ (${\mathcal H}_0:={\Bbb C}$) by
\begin{equation*}
{\mathcal A}_\mu(k):=\left( \begin{array}{cc}
A_{00}(k) & {\mu} A_{01}\\
{\mu} A_{01}^* & A_{11}(k)\\
\end{array}
\right),
\end{equation*}
where the matrix elements are defined by
\begin{align*}
&  A_{00}(k)f_0=w_1(k) f_0,\,\,
A_{01}f_1= \frac{1}{\sqrt{2}}\int_{{\Bbb T}^3} f_1(t)dt,\\
& (A_{11}(k)f_1)(p)=w_2(k,p)f_1(p), \quad f_i \in {\mathcal H}_i, \quad i=0,1.
\end{align*}

Let ${\mathcal A}_0(k):={\mathcal A}_\mu(k)|_{\mu=0}.$ The perturbation ${\mathcal A}_\mu(k)-{\mathcal A}_0(k)$ of the
operator ${\mathcal A}_0(k)$ is a self-adjoint operator of rank 2.
Therefore in accordance with the invariance of the essential
spectrum under the finite rank perturbations the essential
spectrum $\sigma_{\rm ess}({\mathcal A}_\mu(k))$ of ${\mathcal
A}_\mu(k)$ fills the following interval on the real axis
\begin{equation*}
\sigma_{\rm ess}({\mathcal A}_\mu(k))=[m(k); M(k)],
\end{equation*}
where the numbers $m(k)$ and $M(k)$ are defined by
\begin{equation}\label{m(p) and M(p)}
m(k):=\min\limits_{p\in {\Bbb T}^3} w_2(k,p), \quad M(k):=
\max\limits_{p\in {\Bbb T}^3} w_2(k,p).
\end{equation}

From the definition of the function $w_2(\cdot, \cdot)$ we obtain that
this function has an unique non-degenerate
minimum (resp. maximum) at the point $(\overline{0},\overline{0})\in ({\Bbb T}^3)^2$
(resp. $(\overline{\pi}, \overline{\pi})\in ({\Bbb T}^3)^2$) and
$$
\min\limits_{k,p\in {\Bbb T}^3}w_2(k,p)=w_2(\overline{0},\overline{0})=0,\quad
\max\limits_{k,p\in {\Bbb T}^3} w_2(k,p)=w_2(\overline{\pi},\overline{\pi})=18,
$$
where $\overline{0}:=(0, 0, 0),\, \overline{\pi}:=(\pi, \pi, \pi) \in {\Bbb T}^3.$
It is easy to see that
\begin{align*}
& \sigma_{\rm ess}({\mathcal A}_\mu(\overline{0}))=[0; 9{\frac{3}{8}}];\quad
\sigma_{\rm ess}({\mathcal A}_\mu(\bar{\pi}))=[8{\frac{5}{8}}; 18]
\end{align*}
and
$$
\min\{\sigma_{\rm ess}({\mathcal A}_\mu(k)): k\in {\Bbb T}^3\}=0, \quad
\max\{\sigma_{\rm ess}({\mathcal A}_\mu(k)): k\in {\Bbb T}^3\}=18
$$

For any $k\in {\Bbb T}^3$ we define an analytic function $I(k\,;
\cdot)$ in ${\Bbb C} \setminus \sigma_{\rm ess}({\mathcal
A}_\mu(k))$ by
\begin{equation*}
I(k\,; z):=\int_{{\Bbb T}^3} \frac{dt}{w_2(k,t)-z}.
\end{equation*}
The Fredholm determinant associated with the
operator ${\mathcal A}_\mu(k)$ is defined by
\begin{equation*}
\Delta_\mu(k\,; z):=w_1(k)-z-\frac{\mu^2}{2} I(k\,; z),\,\, z\in
{{\Bbb C} \setminus \sigma_{\rm ess}({\mathcal A}_\mu(k))}.
\end{equation*}

From the Birman-Schwinger principle and the Fredholm theorem one can see that \cite{RasDil19}
for any $\mu>0$ and $k\in {\Bbb T}^3$ the operator ${\mathcal A}_\mu (k)$ has an eigenvalue
$z_\mu(k) \in {\Bbb C} \setminus [m(k); M(k)]$ if and
only if $\Delta_\mu(k\,; z_\mu(k))=0.$
Therefore, for the discrete spectrum of ${\mathcal A}_\mu (k)$ the equality
\begin{equation}\label{discAmk}
\sigma_{\rm disc}({\mathcal A}_\mu (k))=\{z \in {\Bbb C} \setminus [m(k); M(k)]:\,
\Delta_\mu(k\,; z)=0 \}
\end{equation}
holds.

\section{Bounds of the essential spectrum of ${\mathcal A}_\mu$}\label{www4}

\indent In this section we study the location of the two- and three-particle branches of the essential spectrum.
We estimate the lower and upper bounds of these branches.

Set
\begin{equation*}
\Lambda_\mu:=\bigcup_{k \in {\Bbb T}^3} \sigma_{\rm disc}({\mathcal A}_\mu (k)), \quad \Sigma_\mu:=[0; 18]
\cup \Lambda_\mu.
\end{equation*}
We recall that
$$\bigcup \limits_{k\in {{\Bbb T}^3}}\sigma_{\rm ess}({\mathcal A}_\mu(k))=[0;18].$$

The following theorem describes the location of
the essential spectrum of the operator ${\mathcal A}_\mu$ by the
spectrum of the family ${\mathcal A}_\mu (k)$ of generalized Friedrichs models.

\begin{theorem}\label{Ess Spectrum THM} For the essential spectrum of ${\mathcal A}_\mu$ the
equality $\sigma_{\rm ess}({\mathcal A}_\mu)=\Sigma_\mu$ holds.
Moreover, the set $\Lambda_\mu$ consists of no more than three
bounded closed intervals.
\end{theorem}
For the proof of Theorem \ref{Ess Spectrum THM} we refer the reader to \cite{AbdLak03}.

In the following we introduce the new subsets (branches) of the essential spectrum of ${\mathcal A}_\mu$:
The sets $\sigma_{\rm two}({\mathcal A}_\mu):=\Lambda_\mu$ and $\sigma_{\rm three}({\mathcal A}_\mu):=[0; 18]$
are called two- and three-particle branches of the essential spectrum of ${\mathcal A}_\mu,$ respectively.

Using the extremal properties of the function $w_2(\cdot,\cdot),$
and the Lebesgue dominated convergence theorem one can show that
the integral $I(\overline{0}\,; 0)$ is finite, see \cite{RasDil19}.
Then in order to estimate the lower and upper bounds of the set $\Sigma_\mu$ we introduce the following quantities
\begin{align*}
& \mu_l^0(\gamma):=\sqrt{2\gamma} \left( I(\overline{0},0) \right)^{-1/2} \quad \mbox{for}\,\, \gamma>0;\\
& \mu_r^0(\gamma):=\sqrt{24-2\gamma} \left( I(\overline{0},0)
\right)^{-1/2}\,\, \mbox{for} \quad \gamma<12;\\
& E_{\mu}^{(1)}:=\min\left\{\Lambda_\mu \cap (-\infty; 0]\right\} \quad \text {for} \, \, \mu\geq \mu_l^0(\gamma);\\
& E_\mu^{(2)}:=\max\left\{ \Lambda_\mu \cap [18; \infty)\right\}  \quad \text {for} \, \,  \mu\geq \mu_r^0(\gamma).
\end{align*}

From the definitions of $\mu_l^0(\gamma)$ and $\mu_r^0(\gamma)$ one can conclude that\\
$(A)$ $\mu_l^0(\gamma)<\mu_r^0(\gamma)$ for $\gamma\in (0; 6);$\\
$(B)$ $\mu_l^0(\gamma)=\mu_r^0(\gamma)$ for $\gamma=6;$\\
$(C)$ $\mu_l^0(\gamma)>\mu_r^0(\gamma)$ for $\gamma\in (6; 12).$

The following theorem gives an full information about the upper bound of the essential spectrum of
${\mathcal A}_\mu$ with respect to the coupling constant $\mu>0$.

\begin{theorem}\label{Structure 1} $(A)$ Let $\gamma<12.$\\
$(A_1)$ For any $\mu \in (0; \mu_r^0(\gamma)],$ the equality $\max\sigma_{\rm ess}({\mathcal A}_\mu)=18$ holds;\\
$(A_2)$ If $\mu >\mu_r^0(\gamma),$ then we have $\max\sigma_{\rm ess}({\mathcal A}_\mu)=E_\mu^{(2)}$ with $E_\mu^{(2)}>18$;\\
$(B)$ Let $\gamma \geq 12.$ Then for any $\mu>0$ the equality $\max\sigma_{\rm ess}({\mathcal A}_\mu)=E_\mu^{(4)}$ holds with $E_\mu^{(2)}>18$.
\end{theorem}

\begin{proof}
$(A)$ Let us consider the case $\gamma<12.$

It is easy to see that $\lim\limits_{z \rightarrow +\infty} \Delta_\mu(k; z)=-\infty$ for any $k\in {\Bbb T}^3.$

$(A_1)$ Suppose that $0<\mu\leq \mu_r^0(\gamma).$ Simple calculation show that $\Delta_\mu(k; z)<0$ for any $k\in {\Bbb T}^3$ and $z>18.$
Since the function $\Delta_\mu(k; \cdot)$ is strictly decreasing in the interval $(18; +\infty),$ by the equality (\ref{discAmk}) for any $k\in{\Bbb T}^3$ the operator ${\mathcal A}_\mu(k)$ has no eigenvalues bigger than 18.
Then by Theorem \ref{Ess Spectrum THM} we obtain $\max\sigma_{\rm ess}({\mathcal A}_\mu)=18.$

$(A_2)$ Assume $\mu>\mu_r^0(\gamma).$ Then $\Delta_\mu(\overline{\pi}; 18)>0.$ Since the function $\Delta_\mu(k; \cdot)$ is strictly decreasing in the interval $(18; +\infty)$ and $\lim\limits_{z \rightarrow +\infty}\Delta_\mu(k; z)=-\infty$ for any $k\in {\Bbb T}^3,$ there exists $z_\mu^0\in (18; +\infty)$ such that $\Delta_\mu(\overline{\pi}; z_\mu^0)=0.$ Therefore, $z_\mu^0$ is an eigenvalue of the operator ${\mathcal A}_\mu(\overline{\pi}).$
It means that $\Lambda_\mu \bigcap (18; +\infty)\neq \emptyset.$ For this case by Theorem \ref{Ess Spectrum THM} we obtain $\max\sigma_{\rm ess}({\mathcal A}_\mu)=E_\mu^{(2)}$ with $E_\mu^{(2)}>18.$

$(B)$ If $\gamma\geq12,$ then it is easy to see that for any $\mu>0$ the inequality $\Delta_\mu(\overline{\pi}; 18)>0$ holds. Then as in the case $(A_2)$ we have
$\max\sigma_{\rm ess}({\mathcal A}_\mu)=E_\mu^{(2)}$ with $E_\mu^{(2)}>18.$
\end{proof}

For the lower bound of the essential spectrum of ${\mathcal A}_\mu$ we formulate the following theorem.

\begin{theorem}\label{structure 2}
$(A)$ If $\gamma\leq 0,$ then for any $\mu>0$ the equality $\min\sigma_{\rm ess}({\mathcal A}_\mu)=E_\mu^{(1)}$ holds with $E_\mu^{(1)}<0.$\\
$(B)$ Let $\gamma>0.$\\
$(B_1)$ If $\mu \in (0; \mu_l^0(\gamma)],$ then we have $\min\sigma_{\rm ess}({\mathcal A}_\mu)=0$;\\
$(B_2)$ Assume $\mu >\mu_l^0(\gamma).$ Then the equality $\min\sigma_{\rm ess}({\mathcal A}_\mu)=E_\mu^{(1)}$ holds with $E_\mu^{(1)}<0.$\\
\end{theorem}
Theorem \ref{structure 2} may be proved in much the same way as Theorem \ref{Structure 1}.

\section{Birman-Schwinger principle}\label{www5}

\indent
In this section we review the corresponding Birman-Schwinger principle for the operator matrix ${\mathcal A}_\mu.$

For any interval $\Delta \subset {\Bbb R},$  $E_{\Delta}({\mathcal A}_\mu)$ stands for the spectral subspace of
${\mathcal A}_\mu$ corresponding to $\Delta.$ Let us denote by
$N_{(a;b)}({\mathcal A}_\mu)$ the number of eigenvalues of the operator ${\mathcal A}_\mu,$ including multiplicities,
lying in $(a; b)\subset {\Bbb R}\setminus \sigma_{\rm ess}({\mathcal A}_\mu),$ that is,
$$
N_{(a; b)}({\mathcal A}_\mu):=\dim E_{(a; b)}({\mathcal A}_\mu){\mathcal H}.
$$

For any $\lambda \in {\Bbb R}$ we define the number follows $n(\lambda, {\mathcal A}_\mu)$ as
$$
n(\lambda, {\mathcal A}_\mu):=\sup\{\dim F:({\mathcal A}_\mu u, u)>\lambda, u\in F\subset {\mathcal H}, ||u||=1 \}.
$$

The number $n(\lambda, {\mathcal A}_\mu)$ is equal to infinity, if $\lambda < \max \sigma_{\rm ess}({\mathcal A}_\mu);$ if $n(\lambda, {\mathcal A}_\mu)$
is finite, then it is equal to the number of the eigenvalues of ${\mathcal A}_\mu$ bigger than $\lambda,$ counted with their multiplicities.

By the definition of $N_{(a;b)}({\mathcal A}_\mu)$ we have
\begin{align*}
& N_{(-\infty; z)}({\mathcal A}_\mu)=n(-z; -{\mathcal A}_\mu), \, \, \, -z>-E_\mu^{(1)};\\
& N_{(z; \infty)}({\mathcal A}_\mu)=n(z; {\mathcal A}_\mu), \, \, \, z>E_\mu^{(2)}.
\end{align*}

We note that for any $k \in {\Bbb T}^3$ and $z<E_\mu^{(1)}$ (resp. $z>E_\mu^{(2)}$) the function
$\Delta_{\mu}(k; z)$ (resp. $-\Delta_{\mu}(k; z)$) is positive and consequently, exists its positive square root.

In our analysis of the discrete spectrum of ${\mathcal A}_\mu$ the crucial role is played by
the following compact operator $T_\mu(z),$ $z\in {\Bbb R}\setminus [E_\mu^{(1)}; E_\mu^{(2)}]$,
acting in the space $L_2({\Bbb T}^3)$ as integral operator
\begin{align*}
& (T_\mu(z)g)(p)=\frac{\mu^2}{2\sqrt{\Delta_{\mu}(p; z)}}\int_{{\Bbb T}^3}\frac{g(t)dt}{\sqrt{\Delta_{\mu}(t; z)}(w_2(p,t)-z)} \quad {\text {for}} \quad z<E_\mu^{(1)},\\
& (T_\mu(z)g)(p)=-\frac{\mu^2}{2\sqrt{-\Delta_{\mu}(p; z)}}\int_{{\Bbb T}^3}\frac{g(t)dt}{\sqrt{-\Delta_{\mu}(t; z)}(w_2(p,t)-z)} \quad {\text {for}}
\quad z>E_\mu^{(2)}.
\end{align*}

The following lemma is a realization of the well known Birman-Schwinger principle
for the operator matrix ${\mathcal A}_\mu.$

\begin{lemma}\label{BSH}
For any $z\in{\Bbb R}\backslash [E_\mu^{(1)}; E_\mu^{(2)}]$ the operator $T_\mu(z)$ is compact and continuous in $z$ and
\begin{align*}
& N_{(-\infty; z)}({\mathcal A}_\mu)=n(1,T_\mu(z)) \quad \text{for} \quad z<E_\mu^{(1)};\\
& N_{(z; \infty)}({\mathcal A}_\mu)=n(1,T_\mu(z)) \quad \text{for} \quad z>E_\mu^{(2)}.
\end{align*}
\end{lemma}

\begin{proof}
The operator ${\mathcal A}_\mu$ can be decomposed as
\begin{equation*}
{\mathcal A}_\mu=\left( \begin{array}{cc}
A_{11} &  0\\
0 & A_{22}\\
\end{array}
\right)+
\mu \left( \begin{array}{cc}
0 &  A_{12}\\
A_{12}^* & 0\\
\end{array}
\right).
\end{equation*}
Denote by $I_i, i=1,2,$ the identity operator on Hilbert space ${\mathcal H}_i, i=1,2$ and by $I={\rm diag}\{I_1, I_2\}$ the identity operator on ${\mathcal H}.$

For any $z<E_\mu^{(1)}$ the operator $A_{ii}-zI_i, i=1,2$ is positive and invertible and hence the square root $R_{ii}^{1\backslash 2}(z)$ of the resolvent $R_{ii}(z)=(A_{ii}-zI_i)^{-1}$ of $A_{ii},$ $i=1,2,$ exists.

Let ${\mathcal M}(z),$ $z<E_\mu^{(1)}$ be the operator matrix with entries
\begin{align*}
& M_{\alpha\alpha}(z)=0, \quad \alpha=1,2,\\
& M_{12}(z):=-{R_{11}^{1/2}(z)}A_{12}R_{22}^{1/2}(z), \quad M_{21}(z):=M_{12}^*(z).
\end{align*}

One has $(({\mathcal A}_\mu-zI)f,f)<0,$ $f=(f_1,f_2)\in{\mathcal H}$ if and only if $((\mu {\mathcal M}(z)-I)g,g)>0,$ $g=(g_1,g_2),$ where $g_i=(A_{ii}-zI_i)^{1/2}f_i,$
$i=1,2.$

It follows that
\begin{equation}\label{BSH1}
N_{(-\infty; z)}({\mathcal A}_\mu)=n(1,\mu {\mathcal M}(z)).
\end{equation}

In the Hilbert space $\mathcal H$ we consider the operator $V_\mu(z):=\mu^2 M_{12}(z)M_{21}(z)$ for $z<E_\mu^{(1)}.$
Denote by $F\subset{\mathcal H}_1$ a subspace for which the equality $\dim F=n(1,V(z))$ holds.
Then
$$
((\mu {\mathcal M}(z)-I)g,g)=((V_\mu(z)-I_1)f_1,f_1)
$$
for all $f_1\in F$ and $g=(f_1, \mu M_{21}(z)f_1).$

Therefore
\begin{equation}\label{BSH2}
n(1,\mu {\mathcal M}(z))=n(1,V_\mu(z)).
\end{equation}
One has $((V_\mu(z)-I_1)\varphi,\varphi)>0,$ $\varphi\in{\mathcal H}_1$ if and only if the inequality
\begin{equation*}
((A_{11}-zI_1)\psi,\psi)<(A_{12}R_{22}(z)A_{12}^*\psi,\psi)
\end{equation*}
holds for $\psi=R_{11}^{1/2}(z)\varphi.$ This means that
\begin{equation}\label{BSH3}
n(1,V_\mu(z))=n(-z,G_\mu(z)),
\end{equation}
where $G_\mu(z):=\mu^2 A_{12}R_{22}(z)A_{12}^*-A_{11}.$

Now we represent the operator $A_{12}^*$ as a sum of two operators $B_1$ and $B_2$ acting from $L_2({\Bbb T}^3)$ to $L_2(({\Bbb T}^3)^2)$ as
\begin{equation*}
(B_1f_1)(p,q)=\frac{1}{2}f_1(q), \quad (B_2f_1)(p,q)=\frac{1}{2}f_1(p), \quad f_1\in L_2({\Bbb T}^3).
\end{equation*}
The operator $D_\mu(z):=A_{11}-z-\mu^2 A_{12}R_{22}(z)B_2$, $z<E_\mu^{(1)}$ is the multiplication operator by the positive function $\Delta_\mu(\cdot; z)$ defined on ${\Bbb T}^3$ and hence it is invertible. It is clear that the positive square root $D_\mu^{-1/2}(z)$ of $D_\mu^{-1}(z)$ is the multiplication operator by the function $\Delta_\mu^{-1/2}(\cdot, z).$

Thus we can conclude that $(G_\mu(z)\varphi,\varphi)>-z(\varphi, \varphi)$ holds if and only if $(T_\mu(z)\eta, \eta)>(\eta, \eta)$ holds for
$\eta=D_\mu^{1/2}(z)\varphi$ and hence,
\begin{equation}\label{BSH4}
n(-z, G_\mu(z))=n(1,T_\mu(z)).
\end{equation}
The equalities (\ref{BSH1}), (\ref{BSH2}), (\ref{BSH3}) and (\ref{BSH4}) give $N_{(-\infty; z)}({\mathcal A}_\mu)=n(1,T_\mu(z)).$

Finally we note that the operator $T_\mu(z),$ $z<E_\mu^{(1)},$ is compact and continuous in $z.$

The same conclusion can be drawn for the case $z>E_\mu^{(2)}.$
\end{proof}

\section{The finiteness of the number of eigenvalues of the operator \\ matrix ${\mathcal A}_\mu$.}

In this section we find conditions which guarantee for the finiteness of the number of eigenvalues of ${\mathcal A}_\mu$ in the left of $E_\mu^{(1)}$
and in the right of $E_\mu^{(2)},$ respectively.
To this end, we need the following three lemmas.

Henceforth, we shall denote by $C_1, C_2, C_3$ different positive numbers. For each $\delta>0$ the notation
$$
U_\delta(p_0):=\{p\in{\Bbb T}^3: |p-p_0|<\delta\}
$$
stands for a $\delta$-neighborhood of the point $p=p_0\in{\Bbb T}^3.$

\begin{lemma}\label{TMF}
There exist positive numbers $C_1, C_2, C_3$ and $\delta>0$ such that the following inequalities hold:\\
$(A)$ $C_1(|p|^2+|q|^2)\leq w_2(p,q)\leq C_2(|p|^2+|q|^2)$ for all $p,q\in U_\delta(\overline{0});$\\
$(B)$ $w_2(p,q)\geq C_3$ for all $(p,q)\not\in U_\delta(\overline{0})\times U_\delta(\overline{0});$\\
$(C)$ $C_1(|p-\overline{\pi}|^2+|q-\overline{\pi}|^2)\leq 18-w_2(p,q)\leq C_2(|p-\overline{\pi}|^2+|q-\overline{\pi}|^2)$
for all $p,q\in U_\delta(\overline{\pi});$\\
$(D)$ $18-w_2(p,q)\geq C_3$ for all $(p,q)\not\in U_\delta(\overline{\pi})\times U_\delta(\overline{\pi}).$
\end{lemma}
\begin{proof}
Since the function $w_2(\cdot, \cdot)$ has an unique non-degenerate minimum (maximum) at the point $(\overline{0},\overline{0})\in({\Bbb T}^3)^2$
($(\overline{\pi},\overline{\pi})\in({\Bbb T}^3)^2$), there exist positive numbers $C_1, C_2, C_3$ and a $\delta$-neighborhood of $p=\overline{0}\in{\Bbb T}^3$
($p=\overline{\pi}\in{\Bbb T}^3$) so that (A) - (D) hold true.
\end{proof}

By the definition of $\Delta_\mu(\cdot;\cdot)$ for any $\mu>\mu_r^0(\gamma)$ with $\gamma<12$ the function $\Delta_\mu(\cdot, E_\mu^{(2)})$ is an analytic function on a compact set ${\Bbb T}^3.$
Therefore, the set $\{p\in{\Bbb T}^3: \Delta_\mu(p, E_\mu^{(2)})=0\}$ is finite.

Denote
$$
\{p\in{\Bbb T}^3: \Delta_\mu(p, E_\mu^{(2)})=0\}=\{p_1, p_2, \ldots, p_N\}, \, \, N<\infty.
$$

The following lemma may be proved in much the same way as Lemma 10 in \cite{LakMum03}.
\begin{lemma}\label{delta}
Let $i\in\{1, 2, \ldots , N\}$ and $\mu>\mu_r^0(\gamma)$ with $\gamma<12.$ Then there exist positive numbers $C_1, C_2, C_3$ and $\delta>0$ such that the following inequalities hold
\begin{equation*}
C_1|p-p_i|^2\leq |\Delta_{\mu}(p; E_\mu^{(2)})|\leq C_2|p-p_i|^2, \quad p\in U_\delta(p_i).
\end{equation*}
\end{lemma}

Now we study some properties of the limit operator $T_\mu(E_\mu^{(2)}).$
\begin{lemma}\label{oper T}
Let $\gamma\geq 12$ and $\mu>0$ be an arbitrary or $\mu\neq \mu_r^0(\gamma)$ for any $\gamma<12.$
Then the operator $T_\mu(z)$ is compact
and continuous in the strong operator topology at the point $z=E_\mu^{(2)}.$
\end{lemma}

\begin{proof}
Let $\gamma\geq 12$ and $\mu>0$ be an arbitrary or $\mu<\mu_r^0(\gamma)$ for any $\gamma<12.$
Under these assumptions by Theorem \ref{Structure 1} we obtain $E_\mu^{(2)}=18.$

Since the function $\Delta_{\mu}(\cdot; E_\mu^{(2)})$ is a positive and continuous on a compact set ${\Bbb T}^3,$ there exist positive numbers $C_1$ and $C_2$
such that
$$
C_1\leq \Delta_{\mu}(p; E_\mu^{(2)})\leq C_2,\quad p\in {\Bbb T}^3.
$$
Then from the statements $(C)$ and $(D)$ of Lemma \ref{TMF} for kernel of the operator $T_\mu(18)$ the following estimate hold
\begin{align*}
\int_{({\Bbb T}^3)^2}&\, \, \frac{dpdq}{\Delta_{\mu}(p; 18) \Delta_{\mu}(q; 18)(w_1(p,q)-18)^2}\\
&\leq C_1\int_{(U_\delta(\overline{\pi}))^2}\frac{dpdq}{(|p-\overline{\pi}|^2+|q-\overline{\pi}|^2)^2}+C_2\\
&=C_1\int_{(U_\delta(\overline{0}))^2}\frac{dpdq}{(p^2+q^2)^2}+C_2.
\end{align*}
Passing on to a spherical coordinate system, then to a polar coordinate system in the last integral we get
\begin{align*}
C_1\int_{(U_\delta(\overline{0}))^2}&\, \, \frac{dpdq}{(p^2+q^2)^2}\leq C_1\int_0^\delta\int_0^\delta \frac{r_1^2r_2^2dr_1dr_2}{(r_1^2+r_2^2)^2}\\
&\leq C_1\int_0^{\pi/2}sin^2 2\varphi d\varphi \int_0^{2\delta}\frac{r^5}{r^4}dr<C_1.
\end{align*}

Let now $\mu>\mu_r^0(\gamma)$ for all $\gamma<12.$ Then from Theorem \ref{Structure 1} we conclude that $E_\mu^{(2)}>18.$
Since the function $w_2(\cdot; \cdot)$ is a continuous on a compact set ${\Bbb T}^3$ and $w_2(p,q)-E_\mu^{(2)}<0$ for any $p,q\in {\Bbb T}^3,$
there exist positive numbers $C_1$ and $C_2,$ such that
\begin{equation}\label{3}
C_1\leq \frac{1}{|w_2(p,q)-E_\mu^{(2)}|}\leq C_2.
\end{equation}

By virtue of estimate (\ref{3}) and Lemma \ref{delta} for kernel of the operator $T_\mu(E_\mu^{(2)})$ we have
\begin{align*}
\int_{({\Bbb T}^3)^2}&\, \, \frac{dpdq}{\Delta_{\mu}(p; E_\mu^{(2)}) \Delta_{\mu}(q; E_\mu^{(2)})(w_2(p,q)-E_\mu^{(2)})^2} \\
&\leq C_1\sum\limits_{i,j=1}^N \int_{U_\delta(p_i)}\frac{dp}{\Delta_{\mu}(p; E_\mu^{(2)})}
\int_{U_\delta(p_j)}\frac{dp}{\Delta_{\mu}(p; E_\mu^{(2)})}+C_2\\
&\leq
C_1\sum\limits_{i,j=1}^N\int_{U_\delta(p_i)}\frac{dp}{|p-p_i|^2}\int_{U_\delta(p_j)}\frac{dp}{|p-p_j|^2}+C_2\\
&=C_1N^2\left(\int_{U_\delta(\overline{0})}\frac{dp}{p^2}\right)^2+C_2.
\end{align*}
Passing on to a spherical coordinate system in the last integral we get
$$
\int_{U_\delta(\overline{0})}\frac{dp}{p^2}\leq C_1\int_0^\delta\frac{r_1^2dr_1}{r_1^2}\leq C_1.
$$

Hence for all $\gamma\geq12$ and $\mu>0$ or $\mu\neq\mu_r^0(\gamma)$ for all $\gamma<12$ the kernel function of the integral operator $T_\mu(z), z\geq E_\mu^{(2)}$ is square-integrable on $({\Bbb T}^3)^2.$  So for any $z\geq E_\mu^{(2)}$ the operator
$T_\mu(z)$ is a Hilbert - Schmidt operator.

It is clear that the kernel of this operator is continuous in $p,q \in {\Bbb T}^3$ for $z>E_\mu^{(2)}.$
Then by the Lebesgue dominated convergence theorem the operator $T_\mu(z)$ is continuous from the right up $z=E_\mu^{(2)}$ in the strong operator topology.
\end{proof}

Now we are ready to the formulate and prove the following result about finiteness of the discrete spectrum of ${\mathcal A}_\mu.$

\begin{theorem}\label{thm1}
$(A)$ Let $\gamma\geq12$ and $\mu>0$ or $\mu\neq\mu_r^0(\gamma)$ with $\gamma<12.$ Then the discrete spectrum of ${\mathcal A}_\mu$ plased to the right
of $E_\mu^{(2)}$ is finite.\\
$(B)$ Let $\gamma\leq0$ and $\mu>0$ or $\mu\neq\mu_l^0(\gamma)$ with $\gamma>0.$ Then the discrete spectrum of ${\mathcal A}_\mu$ placed to the left of $E_\mu^{(1)}$ is finite.
\end{theorem}

\begin{proof}
$(A)$ Suppose that $\gamma \geq 12$ and $\mu>0$  or $\mu \neq \mu_r^0(\gamma)$ for any $\gamma<12.$
By Lemma \ref{BSH} we have
$$
N_{(z; +\infty)}({\mathcal A}_\mu)=n(1; T_\mu(z)) \quad \text{as} \quad z>E_\mu^{(2)},
$$
and by Lemma \ref{oper T} for arbitrary $a\in [0;1)$ the number $n(1-a; T_\mu(E_\mu^{(2)}))$ is finite.

The Weyl inequality
$$
n(\lambda_1+\lambda_2; A_1+A_2)\leq n(\lambda_1; A_1)+n(\lambda_2; A_2)
$$
for the sum of compact operators $A_1$ and $A_2$ and for positive $\lambda_1$ and $\lambda_2$ yields
$$
N_{(z; +\infty)}({\mathcal A}_\mu)=n(1; T_\mu(z))\leq  n(1-a; T_\mu(E_\mu^{(2)}))+n(a; T_\mu(z)-T_\mu(E_\mu^{(2)}))
$$
for arbitrary $z>E_\mu^{(2)}$ and $a \in (0; 1).$

By Lemma \ref{oper T} the operator $T_\mu(z)$ is continuous at $z=E_\mu^{(2)}$ in the strong operator topology, and hence
$$
\lim\limits_{z\rightarrow E_\mu^{(2)}+0}N_{(z; +\infty)}({\mathcal A}_\mu)\leq N_{(E_\mu^{(2)}; +\infty)}({\mathcal A}_\mu)\leq n(1-a; T_\mu(E_\mu^{(2)}))
$$
for arbitrary $a \in (0; 1).$ Consequently
$$
N_{(E_\mu^{(2)}; +\infty)}({\mathcal A}_\mu)\leq n(1-a; T_\mu(E_\mu^{(2)}))<\infty.
$$

The last inequality completes the proof of claim $(A)$ of Theorem \ref{thm1}.

Proof of the assertion $(B)$ of Theorem \ref{thm1} is similar.
\end{proof}

\section{"Two-sided Efimov's effect" for ${\mathcal A}_\mu$ and discrete spectrum asymptotics.}

In this section we prove the existence of the "two-sided Efimov's effect" and we derive the asymptotic relation for the number of eigenvalues of ${\mathcal A}_\mu$ in $(-\infty;0)$ and $(18;+\infty)$ respectively.

Let ${\Bbb S}^2$ being the unit sphere in ${\Bbb R}^3$ and
$$
S_r: L_2((0, r), \sigma_0) \to L_2((0, r), \sigma_0), \quad
r>0, \quad \sigma_0=L_2({\Bbb S}^2)
$$
be the integral operator with the kernel
$$
S(t; y):=\frac{25}{8\pi^2 \sqrt{6}}\, \frac{1}{5 \cosh(y)+t},
$$
$$
y=x-x', \quad x, x' \in (0, r), \quad t=(\xi, \eta), \quad \xi,
\eta \in {\Bbb S}^2.
$$
For $\lambda>0,$ define
$$
U(\lambda)=\frac{1}{2}\lim\limits_{r\rightarrow\infty}r^{-1}n(\lambda, S_r).
$$
The existence of the latter limit and the fact $U(1)>0$ shown in \cite{Sob}.

For completeness, we reproduce the following lemma, which has been
proven in \cite{Sob}.

\begin{lemma}\label{LEM 4.4} Let $A(z):=A_0(z)+A_1(z),$ where $A_0(z)$
$(A_1(z))$ is compact and continuous for $z >18$ $($for $z\geq18).$ Assume that the limit
$$
\lim\limits_{z\to 18+0}f(z)\,n(\zeta, A_0(z)) = l(\zeta)
$$
exists and $l(\cdot)$ is continuous in $(0; +\infty)$ for some
function $f(\cdot),$ where $f(z)\to 0$ as $z \to -0.$ Then the
same limit exists for $A(z)$ and
$$
\lim\limits_{z\to 18+0}f(z)\,n(\zeta, A(z)) = l(\zeta).
$$
\end{lemma}

The following decomposition plays an important role in the proof of the main result.
\begin{lemma}\label{Theorem 3.4}
The following decompositions hold:
\begin{align*}
&\Delta_{\mu_l^0(\gamma)}(k\,; z)=\frac{32 \pi^2 (\mu_l^0(\gamma))^2}
{5\sqrt{5}}\,\sqrt{\frac{6}{5}|k|^2-2z}+ O(|k|^2)+O(|z|),
\end{align*}
with $\gamma>0$ as $|k|\to 0, \, z \nearrow 0; $
\begin{align*}
\Delta_{\mu_r^0(\gamma)}(k\,; z)=&-\frac{32 \pi^2 (\mu_r^0(\gamma))^2}
{5\sqrt{5}}\,\sqrt{\frac{6}{5}|k-\overline{\pi}|^2+2(18-z)}\\
&+ O(|k- \overline{\pi}|^2)+O(|z-18|), \end{align*}
with $\gamma<12$ as $|k-\overline{\pi}|\to
0, \quad z \searrow 18.$
\end{lemma}
Lemma \ref{Theorem 3.4} can be proved similarly to Lemma 4 in \cite{RasDil19} (replacing $\mu_0$ by $\mu_l^0(\gamma)$ with $
\gamma<0$ and $\mu_r^0(\gamma)$ with $\gamma>12,$ respectively). Therefore, to avoid repetition, it is not given here.

One of the main result of this paper is the following theorem.
\begin{theorem}\label{right}
$(A)$ For $\mu=\mu_r^{(0)}(\gamma)$ with $\gamma<12$ the operator ${\mathcal A}_\mu$ has infinitely many eigenvalues lying in
$(18; +\infty)$ and accumulating at $E_\mu^{(2)}=18.$\\
Moreover
\begin{equation}\label{5.7}
\lim\limits_{z \searrow 18}\frac{N_{(z; +\infty)}({\mathcal A}_{\mu_r^0(\gamma)})}{|\log|z-18||}=U(1).
\end{equation}
$(B)$ For $\mu=\mu_l^{(0)}(\gamma)$ with $\gamma>0$ the operator ${\mathcal A}_\mu$ has infinitely many eigenvalues lying in
 $(-\infty; 0)$ and accumulating at $E_\mu^{(1)}=0.$\\
Moreover
$$
\lim\limits_{z \nearrow 0}\frac{N_{(-\infty; z)}({\mathcal A}_{\mu_l^0(\gamma)})}{|\log|z||}=U(1).
$$
\end{theorem}

\begin{proof}
$(A)$ It is clear that by  \eqref{5.7} the infinite cardinality of the discrete spectrum of ${\mathcal A}_{\mu_r^0(\gamma)}$
lying on the r.h.s. of 18 follows automatically from the positivity of $U(1)$. So we derive the asymptotic relation
 (\ref{5.7}) for the number of eigenvalues of ${\mathcal A}_{\mu_r^0(\gamma)},$ bigger than 18.
Now we are going to reduce the study of the asymptotics for the operator $T_{\mu_r^0(\gamma)}(z)$  that of the asymptotics
$T^{(1)}(r)$, which have been studied in \cite{AlbLakRas07}.

Let $T(\delta\,; |z-18|): L_2({\Bbb T}^3) \to L_2({\Bbb T}^3)$ be the integral operator with the kernel
$$
\frac{5\sqrt{5}}{8 \pi^2}\,\, \frac{\chi_\delta(p-
\overline{\pi})\chi_\delta(q- \overline{\pi})(\frac{6}{5}|p-
\overline{\pi}|^2+ 2|z-18|)^{-\frac{1}{4}} (\frac{6}{5}|q-
\overline{\pi}|^2+2|z-18|)^{-\frac{1}{4}}} {5|p-
\overline{\pi}|^2+ 2(p- \overline{\pi}, q- \overline{\pi})+5|q-
\overline{\pi}|^2+8|z-18|},
$$
where $\chi_\delta(\cdot)$ is the characteristic function of the domain
$U_\delta(\overline{0}).$

Applying Lemma \ref{Theorem 3.4} and Lemma \ref{TMF},
one can establish that the operator  $T_{\mu_r^0(\gamma)}(z)-T(\delta\,; |z-18|)$
belongs to the Hilbert--Schmidt class for all $z\geq18$ and small $\delta>0.$
In combination with the continuity of the kernel of the operator with respect to $z>18,$
this implies the continuity of $T_{\mu_r^0(\gamma)}(z)-T(\delta\,; |z-18|)$ in the strong operator topology with respect to $z\geq18.$

From the definition of the kernel function of $T(\delta\,; |z-18|)$ it follows that the subspace of functions
$f$ supported by the set $U_\delta(\overline{\pi})$ is invariant with respect to the operator
$T(\delta\,; |z-18|).$ Let $T^{(0)}(\delta\,; |z-18|)$ be the restriction of the operator $T(\delta\,; |z-18|)$ to the subspace
$L_2(U_\delta(\overline{\pi})),$ that is, the integral operator with the kernel
$$
\frac{5\sqrt{5}}{8 \pi^2}\,\, \frac{(\frac{6}{5}|p-
\overline{\pi}|^2+ 2|z-18|)^{-\frac{1}{4}} (\frac{6}{5}|q-
\overline{\pi}|^2+2|z-18|)^{-\frac{1}{4}}} {5|p-
\overline{\pi}|^2+ 2(p- \overline{\pi}, q- \overline{\pi})+5|q-
\overline{\pi}|^2+8|z-18|}.
$$

Now we consider the following unitary dilation
$$
B_r: L_2(U_\delta( \overline{\pi})) \to L_2(U_r(\overline{0})),\quad (B_r
f)(p)=r^{-\frac{3}{2}}f(\frac{\delta}{r}(p- \overline{\pi})).
$$
Then one can see that the operator $T^{(0)}(\delta\,; |z-18|)$
is unitarily equivalent to the integral operator $T^{(1)}(r),$
acting in $L_2(U_r(\overline{0}))$ with the kernel
$$
\frac{5\sqrt{5}}{8 \pi^2}\,\, \frac{(\frac{6}{5} p^2+
2)^{-\frac{1}{4}} (\frac{6}{5} q^2+2)^{-\frac{1}{4}}} {5 p^2+ 2(p,
q)+5 q^2+8}.
$$

In \cite{AlbLakRas07} it was shown that
$$
\lim_{r \to \infty} |\log r|^{-1} n(1, T^{(1)}(r))=U(1).
$$

Now proof of the statement $(A)$ of Theorem \ref{right} follows from Lemmas \ref{BSH} and \ref{LEM 4.4}.

Proof of the assertion $(B)$ of Theorem \ref{right} is similar.
\end{proof}

Let us mention one important consequence of the Theorem \ref{right}.
\begin{corollary} If $\gamma=6,$ then \\
$(A)$ $\mu_l^0(6)=\mu_r^0(6);$\\
$(B)$ $E_{\mu_l^0(6)}=0,$ $E_{\mu_r^0(6)}=18;$\\
$(C)$ $
\sharp (\sigma_{\rm disc} ({\mathcal A}_{\mu_l^0(6)}) \cap (-\infty, 0))=
\sharp (\sigma_{\rm disc} ({\mathcal A}_{\mu_r^0(6)}) \cap (18,
\infty))=\aleph_0;
$\\
$(D)$
$
\lim\limits_{z \nearrow 0}|\log|z||^{-1}N_{(-\infty;\,
z)}({\mathcal A}_{\mu_r^0(6)})= \lim\limits_{z \searrow 18}
|\log|z-18||^{-1}N_{(z;\, \infty)}({\mathcal A}_{\mu_l^0(6)})=U(1),
$\\
where $\sharp\{\cdot\}$ is a cardinality of a set.
\end{corollary}

{\bf Acknowledgements.} The author would like to thank Dr. Tulkin Rasulov for helpful discussions about
the results of the paper.

\end{document}